\documentclass{article}
\usepackage{amsfonts}
\usepackage{amsmath}
\usepackage{amssymb}
\usepackage{epsfig} 
\usepackage{theorem}

\newtheorem{theorem}{Theorem}
\newtheorem{proposition}[theorem]{Proposition}
\newtheorem{lemma}[theorem]{Lemma} 
\newtheorem{corollary}[theorem]{Corollary}

\theorembodyfont{\rmfamily}
\newtheorem{example}[theorem]{Example}
\newtheorem{definition}[theorem]{Definition}
\newtheorem{remark}[theorem]{Remark}
\newtheorem{conjecture}[theorem]{Conjecture}
\newcommand{\CAT}{{\mathrm{CAT}}}

\newcommand{\algrk}{{\mathrm{rank}}}

\title{Algebraic Rank of $\CAT(0)$ Groups \thanks{partially funded by NSF grant DMS-1207782. }}

\author{Raeyong Kim}

\date{\today}

\newenvironment{proof}[1][]{\begin{trivlist} \item[\hskip\labelsep
\emph{Proof#1.}]}{\foorp \end{trivlist}}%    Proof
\newcommand{\foorp}{{\unskip\nobreak\hfil\penalty50
 \hskip1em\vadjust{}\nobreak\hfil \vrule height3pt width3pt depth0pt
 \parfillskip=0pt \finalhyphendemerits=0 \par}}
\begin{document} 

\maketitle

\begin{abstract} 

We study the algebraic rank of various classes of $\mathrm{CAT}(0)$ groups. They include right-angled Coxeter groups, right-angled Artin groups, relatively hyperbolic groups and groups acting geometrically on $\CAT(0)$ spaces with isolated flats. As one of our corollaries, we obtain a new proof of a result on  commensurability of Coxeter groups. 

\end{abstract}

%%%%%%%%%%%%%%%%%%%%%%%%%%%%%%%%%%%%%%%%%%%%%%%%%%%%%%%%%%%%%%%%%%%%%%%%%%%%%%%%%%%%%%%%%%%%%%%%%%%%%%%%%%%%%%%%%%%%%%%%%%%%%%%%%%%%%%%%%%%%%%%%%%%%%%%%%%%%%%%%%%

\section{Introduction}

Let $M$ be a complete Riemannian manifold of nonpositive sectional curvature. The \emph{geometric rank} of a geodesic $\gamma$ in $M$, denoted by $rk(\gamma)$, is the dimension of the vector space of parallel Jacobi fields along $\gamma$. Then the \emph{geometric rank} of $M$ is defined to be the minimum of $rk(\gamma)$ over all geodesics $\gamma$ in $M$.

The celebrated rank-rigidity theorem, due to Ballmann (\cite{Bal}), and to Burns-Spatzier (\cite{BurSpa}), states that if $M$ has bounded nonpositive sectional curvature and finite volume, then the universal cover $\widetilde{M}$ is a flat Euclidean space, a symmetric space of non-compact type, a space of rank $1$ or a product of such spaces.

In \cite{PraRag}, Prasad and Raghunathan introduced the notion of the algebraic rank, $\mathrm{rank}(G)$, of a group $G$. (See Section 2 for the definition.) Ballmann and Eberlein proved that if $\Gamma$ is the fundamental group of a complete Riemannian manifold $M$ of bounded nonpositive sectional curvature and of finite volume, then $\mathrm{rank}(\Gamma)$ is equal to the geometric rank of $M$ (see \cite{BalEbe}). By combining this with the rank-rigidity theorem, we have that, for a complete Riemannian manifold $M$ of bounded nonpositive sectional curvature and of finite volume, if $\widetilde{M}$ does not have an Euclidean factor and $\Gamma=\pi_{1}(M)$ has higher algebraic rank, either (1) $\Gamma$ is a lattice in a semi-simple Lie group of higher rank, (2) it has a finite index subgroup which splits as a direct product, i.e., $\Gamma$ is a virtually product, or (3) it acts on a product without being a virtual product.

There is an analogous notion of geometric rank for $\mathrm{CAT}(0)$ spaces. A \emph{geometric flat of dimension $n$} in a complete $\mathrm{CAT}(0)$ space $X$ is a closed convex subset of $X$ which is isometric to the Euclidean $n$-space. A geodesic line $L$ is said to have \emph{rank one} if it does not bound a flat half-plane. A complete $\mathrm{CAT}(0)$ space $X$ is said to have \emph{higher geometric rank} if no geodesic in $X$ has rank one.

Let $X$ be a complete $\mathrm{CAT}(0)$ space and $G$ be a group acting geometrically (i.e., properly and cocompactly by isometries) on $X$. In view of the Ballmann-Eberlein's result, it is natural to ask the similar question for $\mathrm{CAT}(0)$ spaces: 

\begin{conjecture}
$G$ has higher algebraic rank if and only if $X$ has higher geometric rank. 
\end{conjecture}

In this paper, we study the algebraic rank of various $\CAT(0)$ groups. They include right-angled Coxeter groups, right-angled Artin groups, relatively hyperbolic groups and groups acting geometrically on $\CAT(0)$ spaces with isolated flats. In Section 3, we prove that if $W$ is an infinite irreducible non-affine right-angled Coxeter group, then $\mathrm{rank}(W) =1$. As a corollary, we obtain a new proof for the question posed by M. Davis in \cite{Dav}, namely, $W$ cannot be commensurable to any uniform lattice in a higher rank non-compact connected semi-simple Lie group. In Subsection \ref{raag}, we prove that any non-join right-angled Artin group has an algebraic rank of $1$. In Section \ref{relhypgppre}, we study algebraic rank of groups which act on $\mathrm{CAT}(0)$ spaces with isolated flats. More precisely, if a group $G$ acts geometrically on $\mathrm{CAT}(0)$ space with isolated flats $\mathcal{F}$ and $|\mathcal{F}| \neq 1$, then $\mathrm{rank}(G) \leq 1$. It is well known that such a group $G$ is hyperbolic relative to a family of stabilizers of flats in $\mathcal{F}$. We use the dynamics of a relatively hyperbolic group acting on the boundary of a $\delta$-hyperbolic space to prove that the algebraic rank of a relative hyperbolic group is $\leq 1$ if there are at least two peripheral subgroups containing elements of infinite order. It follows immediately that $\mathrm{rank}(G) \leq 1$.

This paper is part of author's Ph.D. thesis. The author thanks the thesis advisor Jean Lafont for his guidance throughout this research project. The author also thanks Mike Davis for helpful conversations concerning Coxeter groups.

%%%%%%%%%%%%%%%%%%%%%%%%%%%%%%%%%%%%%%%%%%%%%%%%%%%%%%%%%%%%%%%%%%%%%%%%%%%%%%%%%%%%%%%%%%%%%%%%%%%%%%%%%%%%%%%%%%%%%%%%%%%%%%%%%%%%%%%%%%%%%%%%%%%%%%%%%%%%%%%%%%

\section{Algebraic Rank of Groups}\label{algrk}

\begin{definition}
For a given group $G$, let $\mathcal{A}_{i}(G)$ be the set consisting of elements such that the centralizer contains a free abelian subgroup of rank $\leq i$ as a subgroup of finite index. Define $r(G)$ to be the minimum  $i$ such that $G$ can be expressed as the union of finitely many translates of $\mathcal{A}_{i}(G)$. In other words,
$$\displaystyle{r(G) = min\{i \,|\,\, \textrm{there are finitely many elements}\, g_{j} \in G \, \textrm{such that}\,\, G = \bigcup_{j} g_{j}\mathcal{A}_{i}(G)\}}$$
Finally, the \emph{algebraic rank} of $G$, $\mathrm{rank}(G)$, is the supremum of $r(G^{*})$ over all finite index subgroups $G^{*}$ of $G$.
\end{definition}

We allow the possibility that $\mathrm{rank}(G)=0$. For example, if $G$ is a finite group, then $\mathrm{rank}(G)=0$. On the other hand, if $G$ is torsion-free, then $\mathrm{rank}(G) >0$ : suppose that $\mathrm{rank}(G)=0$. In particular, $r(G)=0$. Then the set $\mathcal{A}_{0}(G)$ must be non-empty. But $\mathcal{A}_{0}(G)$ is a subset of the set of finite order elements in $G$. 

We set $\mathrm{rank}(G)=\infty$ if the sets $\mathcal{A}_{i}(G)$ are empty, or if $G$ cannot be covered by  finitely many translates of any of the sets $\mathcal{A}_{i}(G)$. For example, if a group $G$ has an infinitely generated free abelian center, then $\mathrm{rank}(G)=\infty$. In fact, there exist finitely presented examples of such groups. More specifically, Hall obtained in {\cite{Hal}} the existence of a finitely generated group which has infinitely generated free abelian center. Using {\cite{Oul}}, we can obtain a finitely presented group having  infinitely generated free abelian center.

\begin{remark}
$r(G)$ is not necessarily equal to $\mathrm{rank}(G)$. Following \cite{BalEbe}, we present an example of  a group satisfying $r(G) < \mathrm{rank}(G)$. See {\cite[Section 4]{BalEbe}} for more examples.

Let $G$ be the fundamental group of a flat Klein bottle, acting on $\mathbb{E}^{2}$ by isometries. (In the simplest case $G$ is generated by $\phi_{1} : (x,y) \to (x+1,-y)$ and $\phi_{2}:(x,y) \to (x,y+1)$.) The set $\mathcal{A}_{1}(G)$ consists of all elements of $G$ that reverse the orientation of $\mathbb{E}^{2}$. Then $G=\mathcal{A}_{1}(G) \cup \gamma\mathcal{A}_{1}(G)$, for any $\gamma \in \mathcal{A}_{1}(G)$. Therefore, $r(G) < 2=\mathrm{rank}(G) = \mathrm{rank}(\mathbb{E}^{2})$. 
\end{remark}

We close the section by mentioning that algebraic rank of groups behaves well under products and taking finite index subgroups. 

\begin{proposition}{\cite[Proposition 2.1]{BalEbe}}\label{algrkpro}
Let $G$ be an abstract group.
\begin{enumerate}
\item If $G'$ is a finite index subgroup of $G$, then $r(G) \leq r(G')$ and $\mathrm{rank}(G)=\mathrm{rank}(G')$.
\item If $G=G_1 \times \cdots \times G_n$, then 
$$r(G) = \sum_{i=1}^{n} r(G_{i}),\qquad \mathrm{rank}(G) = \sum_{i=1}^{n} \mathrm{rank}(G_{i})$$
\end{enumerate}
\end{proposition}

%%%%%%%%%%%%%%%%%%%%%%%%%%%%%%%%%%%%%%%%%%%%%%%%%%%%%%%%%%%%%%%%%%%%%%%%%%%%%%%%%%%%%%%%%%%%%%%%%%%%%%%%%%%%%%%%%%%%%%%%%%%%%%%%%%%%%%%%%%%%%%%%%%%%%%%%%%%%%%%%%%

\section{Right -Angled Coxeter Groups and Artin Groups}
\subsection{Coxeter Groups}
A \emph{Coxeter system} $(W,S)$ is a group $W$ and a set $S=\{s_1, s_2, \cdots\}$ of generators such that $W$ has the following presentation
$$W=\langle S\, |\, (s_{i}s_{j})^{m_{ij}} =1, s_{i},s_{j} \in S\rangle,$$
where $m_{ii}=1$ and if $i \neq j$, then $m_{ij}=m_{ji}$ is a positive integer $\geq 2$ or $\infty$ (in which case we omit the relation between $s_{i}$ and $s_{j}$). $W$ is called a \emph{Coxeter group}. A Coxeter system $(W,S)$ is called \emph{irreducible} if $S$ cannot be partitioned into two nonempty disjoint subsets $S'$ and $S''$ such that each element in $S'$ commutes with each element in $S''$. The cardinality $|S|$ of $S$ is called the \emph{rank} of $W$ and we assume that $|S|$ is finite in this section. A Coxeter group $W$ is \emph{spherical} if $W$ is finite and \emph{affine} if $W$ has a finite index free abelian subgroup. 

For any subset $J \subset S$, we denote by $W_{J}$ the subgroup of $W$ generated by $J$. We call $W_{J}$  \emph{a standard parabolic subgroup}, and any conjugate of a standard parabolic subgroup is called \emph{a parabolic subgroup}. For any subset $A \subset W$, the \emph{parabolic closure} $\mathrm{Pc}(A)$ of $A$ is the smallest parabolic subgroup containing $A$. An element $\gamma$ is called \emph{essential} if $\mathrm{Pc}(\gamma)=W$.

Associated to any Coxeter group $W$, there is a $\mathrm{CAT(0)}$ polyhedral cell complex $\Sigma_{W}$, which is called  the \emph{Davis complex}, upon which $W$ acts properly discontinuously and cocompactly by isometries. $\Sigma_{W}$ can be cellulated by so called \emph{Coxeter polytopes} and, with a natural Euclidean metric on each Coxeter polytope, inherits a piecewise Euclidean metric. It was Gromov\,(right-angled case, \cite{Gro}) and Moussong\,(general case, \cite{Mou}), who showed that $\Sigma_{W}$, with this metric, is $\mathrm{CAT}(0)$. See \cite{DavBOOK} for details. An element $\gamma \in W$ is said to have \emph{rank one} if it is hyperbolic and if some (and hence any) of its axes in $\Sigma_{W}$ has rank one. In \cite{CapFuj}, Caprace and Fujiwara study rank one elements in Coxeter groups. In particular, an element $\gamma$ has rank one if and only if its centralizer is virtually infinite cyclic. In other words, any rank one element is contained in $\mathcal{A}_{1}(W)$.
 
Right-angled Coxeter groups are Coxeter groups for which $m_{ij}=2$ or $\infty$ for $i \neq j$. In this case, the Davis complex $\Sigma_{W}$ is a $\mathrm{CAT}(0)$ cubical complex. In this subsection, we prove that any infinite irreducible non-affine right-angled Coxeter group has an algebraic rank of $1$. 

\begin{remark}
\begin{enumerate}
\item Any spherical Coxeter group has an algebraic rank of $0$. (See Section \ref{algrk}.)
\item It is a consequence of Selberg's lemma that every infinite Coxeter group $W$ has a torsion-free subgroup of finite index. Therefore, such groups satisfy $\algrk(W) \geq 1$.  
\item If $W$ is infinite, irreducible and affine, then $\algrk(W) = |S|-1$.
\item Suppose that $W$ is infinite and reducible, $W=W_{T_{1}} \times W_{T_{2}} \times \cdots \times W_{T_{n}}$. By Proposition \ref{algrkpro}, $\displaystyle{\algrk(W) =\sum_{i=1}^{n} \algrk(W_{T_{i}})}$. Therefore, $\algrk(W) > 1$ unless $W$ is virtually cyclic.
\end{enumerate}
\end{remark}

Hereafter, we assume that $W$ is an infinite irreducible non-affine right-angled Coxeter group. Tits' solution to the word problem for Coxeter groups states that any two reduced expressions represent the same element in $W$ if and only if one can be transformed into the other by a series of replacements of the alternating subword $st$ by the subword $ts$. (See \cite[Sec. 3.4]{DavBOOK}.) This implies 

\begin{lemma}\label{word}
For $w \in W$, let $S(w)$ be the set of generators appearing in some (and hence any) reduced expression for $w$. If $s \in S(w)$ appears an odd (respectively, even) number of times in some expression for $w$, then $s$ appears an odd (respectively, even) number of times in any expression for $w$.
\end{lemma}

Let $\mathcal{H}$ be the set of rank one elements in $W$. It is not difficult to find rank one elements in $W$. For example, any essential element in $W$ has virtually infinite cyclic centralizer, therefore it has rank one. (See \cite[Corollary 6.3.10]{Kra})

\begin{lemma}
Let $w$ be an element such that some (and hence any) reduced expression for $w$ has the following property : all generators appear, and each generator appears an odd number of times. Then $w$ is essential, and hence, $w$ has rank one.
\end{lemma}

\begin{proof}
Suppose that $\mathrm{Pc}(w) = uW_{J}u^{-1}$ for some $u$ and some $J \subset S$. Suppose that $s \notin J$. Then any reduced expression for words in $uW_{J}u^{-1}$ contains $s$ an even number of times. Lemma \ref{word} gives contradiction and we can conclude that $s \in J$. This proves that $J=S$.
\end{proof}

\begin{proposition}\label{coxrk}
$r(W) \leq 1$.
\end{proposition}
\begin{proof}
Recall $\mathcal{H} \subset \mathcal{A}_{1}(W)$. Let $\mathcal{S} =\{s_{1}\cdots s_{k} \,| \, s_{i} \in S, \mathrm{distinct}, k \leq n\}$, where $n=|S|$. In other words, $\mathcal{S}$ is the set of all possible products of distinct generators. We prove that for any element $t \in W \setminus \mathcal{H}$, there exists $g \in \mathcal{S}$ such that $gt \in \mathcal{H}$.

Let $t \in W \setminus \mathcal{H}$ be given and consider any reduced expression $\mathbf{t}$ for $t$. Multiply $\mathbf{t}$ by all generators appearing an even (including zero) number of times in $\mathbf{t}$.
$$(s_{i_{1}}\cdots s_{i_{n}})\mathbf{t}$$

Then the resulting word, and hence any reduced expression, has the property that all generators appear and each generator appears an odd number of times. Therefore, the element represented by this word is essential, and hence, it has rank one.

$$\displaystyle{W=\mathcal{H} \bigcup (\bigcup_{\alpha \in \mathcal{S}} \alpha \mathcal{H}}).$$

This proves that $r(W) \leq 1$.
\end{proof}

%%%%%%%%%%%%%%%%%%%%%%%%%%%%%%%%%%%%%%%%%%%%%%%%%%%%%
%%%%%%%%%%%%%%%%%%%%%%%%%%%%%%%%%%%%%%%%%%%%%%%%%%%%%

In order to prove that $\algrk(W) =1$, we need to show $r(T) \leq 1$ for any finite index subgroup $T$ of $W$. But it seems that the above argument does not work for $T$. Because, in the proof of Proposition \ref{coxrk}, $(s_{i_{1}}\cdots s_{i_{n}})$ does not necessarily represent an element in $T$. Once one takes powers on $(s_{i_{1}}\cdots s_{i_{n}})$ to get an element in $T$, the argument fails to apply. In particular, if the index $[W:T]$ is even, all generators in $(s_{i_{1}}\cdots s_{i_{n}})^{[W:T]}$ appear an even number of times. As an example, one can consider the commutator subgroup of $W$. Since the commutator subgroup misses all all-odd elements, it does not contain elements of type appeared in Lemma 7. Therefore, we take a different approach to prove $r(T) \leq 1$.

\begin{definition}
Let $\mathbf{w}$ be a reduced word in $S$. For any generator $s$ appearing in $\mathbf{w}$, let
$$\mathbf{w}=\mathbf{w_{0}}s \mathbf{w_{1}} s \cdots s \mathbf{w_{k}} s \mathbf{w_{k+1}},$$
where $\mathbf{w_{i}}$ does not contain $s$ for all $0 \leq i \leq k+1$. Note that $\mathbf{w_0}$ and $\mathbf{w_{k+1}}$ are allowed to be empty, but each $\mathbf{w_{i}} \neq \emptyset$ for $1 \leq i \leq k$.
\begin{enumerate}
\item $\mathbf{w}$ is said to be \emph{$s$-minimal} if each subword $\mathbf{w_{i}}$, for $1 \leq i \leq k$, contains a $s$-blocker, i.e., a generator $s' \in S$ such that $ss' \neq s's$. We consider $\mathbf{w}$ to be vacuously $s$-minimal if $s$ appears only once in $\mathbf{w}$.
\item $\mathbf{w}$ is said to be \emph{$s$-good} if $\mathbf{w}$ is $s$-minimal and $\mathbf{w_{k+1}w_{0}}$ contains a $s$-blocker for $k \geq 1$. In the case that $k=0$, $\mathbf{w}$ is considered to be $s$-good.
\end{enumerate}
\end{definition}

\begin{remark}
Any reduced word $\mathbf{w}$ is $s$-minimal for all generators $s$ appearing in $\mathbf{w}$. For a generator $s$ appearing in $\mathbf{w}$, let
$$\mathbf{w}=\mathbf{w_{0}}s \mathbf{w_{1}} s \cdots s \mathbf{w_{k}} s \mathbf{w_{k+1}}.$$
If $\mathbf{w_{i}}$ does not contain a $s$-blocker for some $1 \leq i \leq k$, then $s \mathbf{w_{i}} s = \mathbf{w_{i}}$, which is a contradiction.
\end{remark}

\begin{lemma}
Let $\mathbf{w}$ be a reduced word which is $s$-good for all $s \in S$. Then the element represented by $\mathbf{w}$ is essential.
\end{lemma}
\begin{proof}
Suppose that $s$ appears once in $\mathbf{w}$. In other words, $\mathbf{w} = \mathbf{w_{0}}s\mathbf{w_{1}}$.
Suppose the element represented by $\mathbf{w}$ is in $u^{-1}W_{J}u$ for some $u \in W$ and $J \subset S$. Then the element represented by $\mathbf{u}\mathbf{w}\mathbf{u}^{-1}$ lies in $W_{J}$, where $\mathbf{u}$ is any reduced expression for $u$. In some (any) reduced expression of $\mathbf{u}\mathbf{w}\mathbf{u}^{-1}$, $s$ appears an odd number of times. In particular, $s$ appears. Therefore, $s \in J$.

Suppose $s$ appears at least twice in $\mathbf{w}$. In other words,
$$\mathbf{w}=\mathbf{w_{0}}s \mathbf{w_{1}} s \cdots s \mathbf{w_{k}} s \mathbf{w_{k+1}},$$
for $k \geq 1$.
Suppose that the element represented by $\mathbf{w}$ is in $u^{-1}W_{J}u$ for some $u \in W$ and $J \subset S$. Then the element represented by 
$$\mathbf{u}\mathbf{w}\mathbf{u}^{-1} = \mathbf{u}\mathbf{w_{0}}s \mathbf{w_{1}} s \cdots s \mathbf{w_{k}} s \mathbf{w_{k+1}}\mathbf{u}^{-1}$$ 
lies in $W_{J}$, where $\mathbf{u}$ is any reduced expression of $u$. 

Assume that $s \notin J$. Then $\mathbf{u}$ must contain $s$. We prove that one of two $s$'s in $\mathbf{u}$ and $\mathbf{u}^{-1}$ cannot be cancelled off. By way of contradiction, let us assume that both can be cancelled.

Since $\mathbf{w}$ is $s$-good, there exists at least one $s$-blocker in $\mathbf{w_{0}}$ or $\mathbf{w_{k+1}}$. Without loss of generality, we assume that $s$-blocker lies in $\mathbf{w_{0}}$. (A symmetric argument applies if it lies in $\mathbf{w_{k+1}}$.) Take the first $s$-blocker in $\mathbf{w_0}$ and call it $s_1$. It follows that $\mathbf{u}$ must contain $s_1$ and the last $s_{1}$ occurs after the last $s$ in $\mathbf{u}$.
$$\mathbf{u}\mathbf{w}\mathbf{u}^{-1} = (\cdots s \cdots s_1 \cdots)(\cdots s_1 \cdots) s \mathbf{w_{1}} s \cdots s \mathbf{w_{k}} s \mathbf{w_{k+1}}(\cdots s_{1} \cdots s \cdots)$$
In order for the last $s$ to be cancelled off, $s_1$ must occur in $\mathbf{w_{k+1}}$.
$$\mathbf{u}\mathbf{w}\mathbf{u}^{-1} = (\cdots s \cdots s_1 \cdots)(\cdots s_1 \cdots) s \mathbf{w_{1}} s \cdots s \mathbf{w_{k}} s (\cdots s_1 \cdots)(\cdots s_{1} \cdots s \cdots),$$
where the $s_1 \in \mathbf{w_{k+1}}$ written above is the last occurrence of $s_1$ in $\mathbf{w_{k+1}}$. So before being able to cancel the first and the last $s$, we need to cancel out the intermediate blocker $s_1$. 
 
Now $\mathbf{w}$ is $s_1$-good. Therefore, there exists an $s_1$-blocker on the left of the first $s_1 \in \mathbf{w_{0}}$ or on the right of the last $s_1 \in \mathbf{w_{k+1}}$. Take the first $s_1$-blocker in $\mathbf{w_{0}}$ or the last $s_1$-blocker in $\mathbf{w_{k+1}}$ and call it $s_2$. Note that $s_2 \neq s$ and before canceling the $s_1$, we must first be able to cancel out the $s_1$-blocker $s_2$. As in the last paragraph, this forces $\mathbf{u}\mathbf{w}\mathbf{u}^{-1}$ to be of the form
$$(\cdots s_1 \cdots s_2 \cdots)(\cdots s_2 \cdots s_1 \cdots) s \mathbf{w_{1}} s \cdots s \mathbf{w_{k}} s (\cdots s_1 \cdots s_2 \cdots)(\cdots s_2 \cdots s_{1} \cdots)$$

Note that $\mathbf{w}$ is $s_2$-good, and hence, there exists an $s_2$-blocker $(\neq s, s_1)$ on the left of the first $s_2$ or on the right of the last $s_2$ in $\mathbf{w}$. But since $|S| < \infty$, this process must stop in finitely many stages, which proves that one of $s$'s in $\mathbf{u}$ and $\mathbf{u}^{-1}$ cannot be cancelled off. Therefore, $s \in J$. The element represented by $\mathbf{w}$ is essential.
\end{proof}

Let $T$ be a proper finite index subgroup of $W$. Assume that $T$ is normal and let $n=[W:T] \geq 2$. In order to prove that $r(T) \leq 1$, we need to consider two types of generators for a given reduced word $\mathbf{w}$ representing an element in $T$ : (1) a generator $s$ does not appear in $\mathbf{w}$ and (2) a generator $s$ appears, but $\mathbf{w}$ is not $s$-good. We begin with generators of type (1).

Let $\mathbf{w}$ be a reduced word in $S$ representing an element in $T$ and assume that $\mathbf{w}$ misses a generator $s$. Choose $s' \in S$ such that $ss' \neq s's$. Note that such $s'$ always exists because $W$ is infinite irreducible. Next choose $s'' \in S$ such that either $ss'' \neq s''s$ or $s's'' \neq s''s'$. Note that such $s''$ always exists, otherwise $W=W_{\{s,s'\}} \times W_{S\setminus\{s,s'\}}$, contradicting irreducibility (if $|S| >2$) or non-affine (if $|S| = 2$).

\begin{lemma}\label{coxlemma1}

\begin{enumerate}
\item Suppose that $ss''\neq s''s$. Then any reduced expression $\mathbf{r}$ of $(s''ss')^{n}\mathbf{w}$ has the following property:
\begin{enumerate}
\item The element represented by $\mathbf{r}$ is in $T$. This is obvious.
\item $\mathbf{r}$ is $s$-good.
\item $\mathbf{r}$ is $s'$-good.
\item For $t \neq s,s',s''$, if $\mathbf{w}$ is $t$-good, then $\mathbf{r}$ is also $t$-good.
\end{enumerate}
\item Suppose that $s's'' \neq s''s'$. Then any reduced expression $\mathbf{r'}$ of $(s's''ss's'')^{n}\mathbf{w}$ has the following property:
\begin{enumerate}
\item The element represented by $\mathbf{r'}$ is in $T$.
\item $\mathbf{r'}$ is $s$-good.
\item $\mathbf{r'}$ is $s''$-good.
\item For $t \neq s,s',s''$, if $\mathbf{w}$ is $t$-good, then $\mathbf{r'}$ is also $t$-good.
\end{enumerate}
\end{enumerate}
\end{lemma}

\begin{proof} 
We prove the first statement only. The second statement can be proved by exactly the same argument as the first.

Consider 
$$(s''ss')^{n}\mathbf{w}=(s''ss')(s''ss')\cdots(s''ss')\mathbf{w}=(s'')s(s's'')s(s'\cdots s'')s(s'\mathbf{w}).$$
Since $ss' \neq s's$, $ss''\neq s''s$, $s'$ or $s''$ before the last occurrence of $s$ cannot be cancelled. Since $s$ does not appear in $\mathbf{w}$, $\mathbf{r}$ is $s$-good.
Let $\mathbf{w}=\mathbf{w'_{0}}s'\mathbf{w'_{1}}s'\cdots \mathbf{w'_{l}}s' \mathbf{w'_{l+1}}$ and consider
$$(s''ss')^{n}\mathbf{w} = (s''s)s'(s''s)s'\cdots s'(s''s)s'\mathbf{w'_{0}}s'\mathbf{w'_{1}}s'\cdots \mathbf{w'_{l}}s' \mathbf{w'_{l+1}}.$$
Note that if the subword $\mathbf{w'_{0}}$ does not contain $s'$-blocker, then $s'\mathbf{w'_{0}}s'$ is reduced to $\mathbf{w'_{0}}$. But $s$ is an $s'$-blocker. Therefore, there exists at least one $s'$-blocker in the reduced expression of $(s''ss')\mathbf{w'_{0}}s'\mathbf{w'_{1}}$, even if $\mathbf{w'_{0}}$ does not contain $s'$-blocker. (Recall that the letter $s$ does not appear in $\mathbf{w'_{0}}$ and $\mathbf{w'_{1}}$, so it cannot be cancelled off.)

Suppose that $\mathbf{w}$ is $t$-good for $t \neq s,s',s''$. If neither $s,s'$ or $s''$ is $t$-blocker, then $\mathbf{r}$ is obviously $t$-good. Consider the case that either $s,s'$ or $s''$ is a $t$-blocker.
$$(s''ss')^{n}\mathbf{w}=(s''ss')(s''ss')\cdots(s''ss')\mathbf{w''_{0}}t\mathbf{w''_{1}}t\cdots \mathbf{w''_{m}}t \mathbf{w''_{m+1}}$$
Note that $s$ is not in $\mathbf{w}$ and $n \geq 2$. Therefore, there exists at least one $t$-blocker in the reduced expression of $(s''ss')(s''ss')\cdots(s''ss')\mathbf{w''_{0}}$. It follows that $\mathbf{r}$ is $t$-good.
\end{proof}

\begin{remark}\label{allgen}
\begin{enumerate}
\item In Lemma \ref{coxlemma1}, $\mathbf{r}$ is not necessarily $s''$-good. Similarly, $\mathbf{r'}$ is not necessarily $s'$-good. Therefore, the number of good generators of $\mathbf{r}$ or $\mathbf{r'}$ might be equal to the number of good generators of $\mathbf{w}$. But note that $s$ appears in $\mathbf{r}$ and $\mathbf{r'}$, and all letters appearing in $\mathbf{w}$ still appear in $\mathbf{r}$ and $\mathbf{r'}$.
\item Let $\mathbf{w}$ be a reduced word in $S$ representing an element in $T$. By multiplying words as in Lemma \ref{coxlemma1}, we can obtain $\mathbf{w':=y_{k}y_{k-1}\cdots y_{1}w}$ such that the element represented by $\mathbf{w'}$ is in $T$ and all generators appear in $\mathbf{w'}$.
\item There exists a finite set $\mathcal{R}$ of words such that for any reduced word $\mathbf{w}$ representing an element in $T$, there exists some $\mathbf{r} \in \mathcal{R}$ for which $\mathbf{rw}$ represents an element in $T$ and all generators appear in $\mathbf{rw}$ .
\end{enumerate}   
\end{remark}

Next, we consider generators of type (2).

\begin{definition}
Let $\mathbf{w}$ be a reduced word in $S$ such that all generators appear. Define $B(\mathbf{w})$ be the set of generators for which $\mathbf{w}$ is not good, i.e., $B(\mathbf{w})$ consists of all the ``bad" generators. For $B \subset S$, a word $\mathbf{v}$ is called a \emph{$B$-cancellator} if, for any $\mathbf{w}$ such that $B=B(\mathbf{w})$, any reduced expression of $\mathbf{v}\mathbf{w}$ is $s$-good for all $s \in S$.
\end{definition}

The following lemma tells us that $B$-cancellators exist and can be chosen to represent an element in $T$. Note that there are only finitely many subsets of $S$. Therefore, we can form finitely many cancellators.

Let $\mathbf{w}$ be a reduced word in $S$ such that all generators appear and $\mathbf{w}$ represents an element in $T$. Suppose that $\mathbf{w}$ is not $s$-good. Choose $s' \in S$ such that $ss' \neq s's$. Note that such $s'$ always exists because $W$ is infinite irreducible. Also we choose $s'' \in S$ such that either $ss'' \neq s''s$ or $s's'' \neq s''s'$. Note that such $s''$ always exists, otherwise $W=W_{\{s,s'\}} \times W_{S\setminus\{s,s'\}}$.

\begin{lemma}\label{coxlemma2}
\begin{enumerate}
\item Suppose that $ss''\neq s''s$. Then any reduced expression $\mathbf{r}$ of $(s''ss')^{n}\mathbf{w}$ has the following property:
\begin{enumerate}
\item The element represented by $\mathbf{r}$ is in $T$. This is obvious.
\item $\mathbf{r}$ is $s$-good.
\item $\mathbf{r}$ is $s'$-good.
\item $\mathbf{r}$ is $s''$-good.
\item For $t \neq s,s',s''$, if $\mathbf{w}$ is $t$-good, then $\mathbf{r}$ is also $t$-good.
\end{enumerate}
\item Suppose that $s's'' \neq s''s'$. Then any reduced expression $\mathbf{r'}$ of $(s's''ss's'')^{n}\mathbf{w}$ has the following property:
\begin{enumerate}
\item The element represented by $\mathbf{r'}$ is in $T$.
\item $\mathbf{r'}$ is $s$-good.
\item $\mathbf{r'}$ is $s'$-good.
\item $\mathbf{r'}$ is $s''$-good.
\item For $t \neq s,s',s''$, if $\mathbf{w}$ is $t$-good, then $\mathbf{r'}$ is also $t$-good.
\end{enumerate}
\end{enumerate}
\end{lemma}

\begin{proof}
Again, we prove the first statement only. The second statement can be proved by exactly the same argument as the first.

Let $\mathbf{w} = \mathbf{w_{0}}s \mathbf{w_{1}} s \cdots s \mathbf{w_{k}} s \mathbf{w_{k+1}}$ and consider
$$(s''ss')^{n}\mathbf{w} = (s'')s(s's'')\cdots(s's'')s(s'\mathbf{w_{0}})s \mathbf{w_{1}} s \cdots s \mathbf{w_{k}} s \mathbf{w_{k+1}}.$$
In the reduced expression of $ss'\mathbf{w_{0}}s$, since $s$ and $s'$ don't commute, $s'$ is $s$-blocker. (Note that $\mathbf{w_{0}}$ does not contain $s'$, since, by assumption, $s \in B(\mathbf{w}))$. Also the first $s''$ is an $s$-blocker. Hence $\mathbf{r}$ is $s$-good.

$$(s''ss')^{n}\mathbf{w}=(s''s)s'(s''s)s'\cdots(s''s)s'\mathbf{w'_{0}}s'\mathbf{w'_{1}}s'\cdots \mathbf{w'_{m}}s' \mathbf{w'_{j+1}}$$
Secondly, note that $s$ and $s'$ don't commute. Since $\mathbf{w}$ is not $s$-good, the first $s'$ in $\mathbf{w}$ should occur after the first $s$ in $\mathbf{w}$, i.e., $s \in \mathbf{w'_{0}}$. It follows that $s$ is an $s'$-blocker in $\mathbf{w'_{0}}$. Hence $\mathbf{r}$ is $s'$-good.

$$(s''ss')^{n}\mathbf{w}=s''(ss')s''(ss')\cdots s''(ss'\mathbf{w''_{0}})s''\mathbf{w''_{1}}s''\cdots \mathbf{w''_{i}}s'' \mathbf{w''_{i+1}}$$

Note that $s$ and $s''$ don't commute. Since $\mathbf{w}$ is not $s$-good, the last $s''$ in $\mathbf{w}$ should occur before the last $s$ in $\mathbf{w}$, i.e., $s \in \mathbf{w''_{i+1}}$. It follows that $\mathbf{w''_{i+1}}$ contains an $s''$-blocker $s$. $\mathbf{r}$ is $s''$-good.

Suppose that $\mathbf{w}$ is $t$-good for $t \neq s, s', s''$. If neither $s,s'$ or $s''$ is $t$-blocker, then $\mathbf{r}$ is $t$-good. Let 
$$(s''ss')^{n}\mathbf{w}=(s''ss')(s''ss')\cdots(s''ss')\mathbf{w'''_{0}}t\mathbf{w'''_{1}}t\cdots \mathbf{w'''_{h}}t \mathbf{w'''_{h+1}}.$$

If $s$ is a $t$-blocker, the first $t$ occur after the first $s$ in $\mathbf{w}$, i.e., $s \in \mathbf{w'''_{0}}$. Furthermore, this $s$ cannot be cancelled off, because $s$ and $s'$ don't commute. It follows that there exists at least one $t$-blocker in the reduced expression of $(s''ss')\mathbf{w'''_{0}}$. $\mathbf{r}$ is $t$-good. Next, suppose that $s'$ or $s''$ is $t$-blocker. The $s$ in the last $(s''ss')$ cannot be cancelled off. Therefore, there exists at least one $t$-blocker in the reduced expression of $(s''ss')(s''ss')\cdots(s''ss')\mathbf{w'''_{0}}$. $\mathbf{r}$ is $t$-good.
\end{proof}

\begin{corollary}\label{cancelator}
For a given $B \subset S$, a $B$-cancellator $\mathbf{v_{B}}$ exists and can be chosen to represent an element in $T$.
\end{corollary}

\begin{proof}
Let $\mathbf{w}$ be a reduced word in $S$, with all generators appearing, such that $B(\mathbf{w})=B$. Choose a generator $s_{1} \in B(\mathbf{w})$. Apply the lemma to obtain a word $\mathbf{v_{1}}$ representing an element in $T$ such that any reduced expression $\mathbf{r}_{1}$ of $\mathbf{v_{1}w}$ is $s_{1}$-good. Consider $\mathbf{r}_{1}$. From Lemma \ref{coxlemma2}, $B(\mathbf{r}_{1}) \subset B(\mathbf{w})$ and $|B(\mathbf{r}_{1})| < |B(\mathbf{w})|$. Choose a generator $s_{2} \in B(\mathbf{r}_{1})$ and apply the lemma to obtain a word $\mathbf{v_{2}}$ representing an element in $T$ such that any reduced expression $\mathbf{r}_{2}$ of $\mathbf{v_{2}}\mathbf{r}_{1}$ is $s_{2}$-good. Continuing this process, at most $|B(\mathbf{w})|$ number of times, we obtain a word $\mathbf{v_{k}}\mathbf{v_{k-1}}\cdots \mathbf{v_{1}} \mathbf{w}, k \leq |B(\mathbf{w})|$ whose reduced expression is $s$-good for all $s \in S$. Such a $\mathbf{v_{k}}\mathbf{v_{k-1}}\cdots \mathbf{v_{1}}$ is the desired $B$-cancellator.
\end{proof}

\begin{corollary}\label{norcox}
$r(T) \leq 1$.
\end{corollary}
\begin{proof}
Let $\mathcal{H}_{T} = \mathcal{H} \bigcap T$. For $g \in \mathcal{H}_{T}$, the centralizer $C_{W}(g)$ in $W$ is virtually infinite cyclic, and hence, $C_{T}(g)$ is also virtually infinite cyclic. It follows that $\mathcal{H}_{T} \subset \mathcal{A}_{1}(T)$. 

Let $g' \in T \setminus \mathcal{H}_{T}$ be given. By Remark \ref{allgen} and Corollary \ref{cancelator}, we can find a $\mathbf{r} \in \mathcal{R}$ and a cancellator $\mathbf{v}$ such that any reduced expression of $\mathbf{v}\mathbf{rw'}$ is $s$-good for all $s \in S$, where $\mathbf{w'}$ is any reduced expression for $g'$. Note that $\mathbf{v}\mathbf{rw'}$ represents an element in $\mathcal{H}_{T}$. Therefore,

$$T = \bigcup_{B \subseteq S, \mathbf{r} \in \mathcal{R}} \mathbf{r}^{-1}\mathbf{v_B}^{-1}\mathcal{H}_{T},$$
where each $\mathbf{v}_{B}$ is a $B$-cancellator.
\end{proof}

\begin{proposition}\label{cox}
Let $W$ be an infinite, irreducible, and non-affine right-angled Coxeter group. Then $\algrk(W)=1$.
\end{proposition}

\begin{proof}
By Proposition \ref{coxrk}, it suffices to show that $r(W') \leq 1$ for any finite index subgroup $W'$ of $W$. By taking the normal core, we obtain a finite index normal subgroup $W''$ of $W$ and, by Corollary \ref{norcox}, $r(W'') \leq 1$. Finally, Proposition \ref{algrkpro} implies that $r(W') \leq 1$.
\end{proof}

%%%%%%%%%%%%%%%%%%%%%%%%%%%%%%%%%%%%%%%%%%%%%%%%%%%%%%
%%%%%%%%%%%%%%%%%%%%%%%%%%%%%%%%%%%%%%%%%%%%%%%%%%%%%%

We close the subsection by introducing two corollaries of Proposition \ref{cox}. Two groups $G_1$ and $G_2$ are said to be \emph{commensurable} if there exist $H_{i}$, for $i=1,2$, such that $[G_{i}:H_{i}] < \infty $ or $[H_{i}:G_{i}] < \infty$ and $H_{1}$ is isomorphic to $H_{2}$. Davis proved that any infinite irreducible non-affine Coxeter group $W$ cannot be a uniform lattice $\Gamma$ in a higher rank non-compact connected semi-simple Lie group (See \cite[Corollary 10.9.8]{DavBOOK}) and conjectured $W$ cannot be commensurable to $\Gamma$. (See \cite{Dav}) The conjecture is known to be true, see for example, \cite{CooLon}, \cite{Gon}, \cite{Sin}. Proposition \ref{cox} provides a new proof of the conjecture for right-angled Coxeter groups.

\begin{corollary}\label{commen}
Let $W$ be an infinite irreducible non-affine right-angled Coxeter group. Then $W$ is not commensurable to any uniform lattice in a higher rank non-compact connected semi-simple Lie group $G$.
\end{corollary}

\begin{proof}
Let $\Lambda$ be a uniform lattice in $G$. By \cite[Theorem 3.11]{BalEbe}, $\algrk(\Lambda) = \algrk(G) \geq 2$. Applying Proposition \ref{algrkpro}, we obtain that any group $\Gamma$ commensurable to $\Lambda$ satisfies $\algrk(\Gamma) = \algrk(\Lambda) \geq 2$. On the other hand, for any finite index subgroup $W'$ of $W$, $\algrk(W)=\algrk(W')=1$. Therefore, $W$ and $\Lambda$ cannot be commensurable.
\end{proof}

%Two metric spaces $X$ and $Y$ are said to be \emph{quasi-isometric} if there exists a map $f : X \to Y$ such that, for some $\lambda \geq 1, \epsilon \geq 0$ and $C \geq 0$, such that 
%\begin{enumerate}
%\item For all $x_1, x_2 \in X$, 
%$$\frac{1}{\lambda}d_{X}(x_1,x_2) - \epsilon \leq d_{Y}(f(x_{1}), f(x_{2})) \leq \lambda d_{X}(x_1,x_2) + \epsilon.$$
%\item $Y$ lies in the $C$-neighborhood of the image of $f$.
%\end{enumerate}
The other corollary follows from Quasi-Isometry rigidity theorem due to Kleiner-Leeb (\cite{KleLee}) and Eskin-Farb (\cite{EskFar}).

\begin{corollary}
Let $W$ be an infinite irreducible non-affine right-angled Coxeter group. Then $W$ is not quasi-isometric to any uniform lattice in a higher rank non-compact connected semi-simple Lie group $G$.
\end{corollary}

\begin{proof}
By QI-rigidity theorem, if $W$ is quasi-isometric to a uniform lattice in a higher rank non-compact connected semi-simple Lie group, $W$ should be commensurable to a lattice. By Corollary \ref{commen}, $W$ cannot be commensurable to the lattice.
\end{proof}

%%%%%%%%%%%%%%%%%%%%%%%%%%%%%%%%%%%%%%%%%%%%%%%%%%%%%%%%%%%%%%%%%%%%%%%%%%%%%%%%%%%%%%%%%%%%%%%%%%%%%%%%%%%%%%%%%%%%%%%%%%%%%%%%%%%%%%%%%%%%%%%%%%%%%%%%%%%%%%%%%%

\subsection{Algebraic Rank of Right-Angled Artin Groups}\label{raag}

An Artin group $A$ is a group with the following presentation :

$$A=\langle s_{1}, \cdots, s_{n} | \underbrace{s_{i}s_{j}\cdots}_{m_{ij}} = \underbrace{s_{j}s_{i}\cdots}_{m_{ij}}, \mathrm{\,\,for\,\,all\,\,} i \neq j \rangle,$$ 
where $m_{ij}=m_{ji}$ is an integer $\geq 2$ or $m_{ij}=\infty$ in which case we omit the relation between $s_{i}$ and $s_{j}$. As one can see, by adding relations $s_{i}=s_{i}^{-1}$ to the presentation, we obtain a Coxeter group. Right-angled Artin groups are those Artin groups for which all $m_{ij}=2$ or $\infty$ for $i \neq j$.

One of the easy ways of defining a right-angled Coxeter group or a right-angled Artin group is via \emph{the defining graph} $\Gamma$. This is the graph whose vertices are labeled by $S=\{s_{1}, \cdots, s_{n}\}$ and two vertices $s_{i}$ and $s_{j}$ are connected if $m_{ij}=2$. We denote by $A_{\Gamma}\, (W_{\Gamma}$, respectively) the right-angled Artin group (the right-angled Coxeter group, respectively) associated to a finite simplicial graph $\Gamma$. For example, if $\Gamma$ consists of $n$ vertices and no edges, then $A_{\Gamma}$ is the free group on $n$ generators. At the other extreme, if $\Gamma$ is a complete graph with $n$ vertices, $A_{\Gamma}$ is the free abelian group of rank $n$. 

Analogous to the Coxeter group situation, there is a $\mathrm{CAT}(0)$ space associated to a right-angled Artin group $A_{\Gamma}$, which can be constructed by the following process : begin with a wedge of circles attached to a point $x_{0}$ and labeled by the generators $s_{1}, \cdots s_{n}$. For each edge connecting $s_{i}$ and $s_{j}$ in $\Gamma$, attach a $2$-torus with boundary labeled by the relator $s_{i}s_{j}s_{i}^{-1}s_{j}^{-1}$. For each triangle connecting $s_{i},s_{j},s_{k}$ in $\Gamma$, attach a $3$-torus with faces corresponding to the tori for the three edges of triangle. Continuing this process, attach a $k$-torus for each set of $k$-mutually commuting generators. The resulting cube complex is called \emph{a Salvetti complex} for $A_{\Gamma}$ and denoted by $\mathcal{S}_{\Gamma}$. It is easy to verify that the fundamental group of $\mathcal{S}_{\Gamma}$ is $A_{\Gamma}$ and the link of the unique vertex $x_{0}$ is a flag. It follows from Gromov's criterion that the universal cover $X_{\Gamma}$ of the complex $\mathcal{S}_{\Gamma}$ is a $\mathrm{CAT}(0)$ cube complex, and $A_{\Gamma}$ acts on $X_{\Gamma}$ freely and cocompactly. 

Given two graphs $\Gamma_{1}, \Gamma_{2}$, their \emph{join} is the graph obtained by connecting every vertex of $\Gamma_{1}$ to every vertex of $\Gamma_{2}$. If $\Gamma$ is the join of $\Gamma_{1}$ and  $\Gamma_{2}$, then $A_{\Gamma} = A_{\Gamma_{1}} \times A_{\Gamma_{2}}$ and $X_{\Gamma} = X_{\Gamma_{1}} \times X_{\Gamma_{2}}$. We prove 

%%%%%%%%%%%%%%%%%%%%%%%%%%%%%%%%%%%%%%%%%%%%%%%%%%%%%

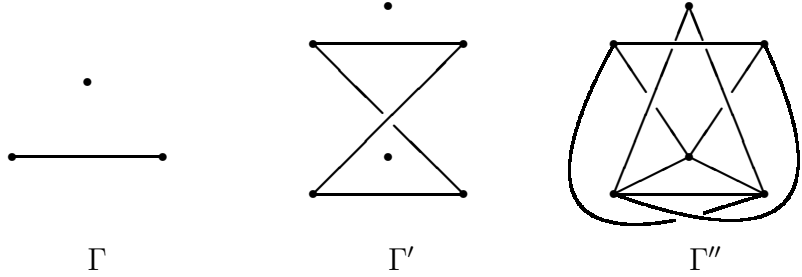
\setlength{\unitlength}{1cm}
\begin{picture}(11,4.2)
\thicklines

\put(1.5,0.5){\large{$\Gamma$}}
\put(0.5,2){\circle*{.1}}
\put(2.5,2){\circle*{.1}}
\put(1.5,3){\circle*{.1}}
\put(0.5,2){\line(1,0){2}}

\put(5.5,0.5){\large{$\Gamma'$}}
\put(4.5,1.5){\circle*{.1}}
\put(6.5,1.5){\circle*{.1}}
\put(5.5,2.0){\circle*{.1}}
\put(4.5,1.5){\line(1,0){2}}
\put(4.5,3.5){\circle*{.1}}
\put(6.5,3.5){\circle*{.1}}
\put(5.5,4.0){\circle*{.1}}
\put(4.5,3.5){\line(1,0){2}}
\put(4.5,1.5){\line(1,1){2}}
\put(6.5,1.5){\line(-1,1){0.92}}
\put(4.5,3.5){\line(1,-1){0.92}}

\put(9.5,0.5){\large{$\Gamma''$}}
\put(8.5,1.5){\circle*{.1}}
\put(10.5,1.5){\circle*{.1}}
\put(9.5,2.0){\circle*{.1}}
\put(8.5,1.5){\line(1,0){2}}
\put(8.5,1.5){\line(2,1){1}}
\put(10.5,1.5){\line(-2,1){1}}
\put(8.5,3.5){\circle*{.1}}
\put(10.5,3.5){\circle*{.1}}
\put(9.5,4.0){\circle*{.1}}
\put(8.5,3.5){\line(1,0){2}}
\put(9.5,2.0){\line(2,3){0.43}}
\put(9.5,2.0){\line(-2,3){0.43}}
\put(10.5,3.5){\line(-2,-3){0.43}}
\put(8.5,3.5){\line(2,-3){0.43}}
\put(8.5,1.5){\line(2,5){0.77}}
\put(9.5,4.0){\line(-2,-5){0.18}}
\put(10.5,1.5){\line(-2,5){0.77}}
\put(9.5,4.0){\line(2,-5){0.18}}

\qbezier(8.5,1.5)(12,0.25)(10.5,3.5)
\qbezier(9.3,1.15)(7,0.75)(8.5,3.5)
\qbezier(10.5,1.5)(9.7,1.25)(9.7,1.25)

\end{picture}

%%%%%%%%%%%%%%%%%%%%%%%%%%%%%%%%%%%%%%%%%%%%%%%%%%%%%

\begin{proposition}\label{art}
If $\Gamma$ is not a join, then $\mathrm{rank}(A_{\Gamma})=1$.
\end{proposition}

\begin{remark}
If $\Gamma$ is the join of $\Gamma_{1}$ and $\Gamma_{2}$, then
$$\mathrm{rank}(A_{\Gamma}) = \mathrm{rank}(A_{\Gamma_{1}}) + \mathrm{rank}(A_{\Gamma_{2}}) \geq 2.$$
\end{remark}

The proof of Proposition \ref{art} is a direct consequence of a theorem due to Davis and Januszkiewicz(\cite{DavJan}).
For a given graph $\Gamma$, we define two graphs $\Gamma '$ and $\Gamma''$ as follows : The vertex set of $\Gamma ''$ is $I \times \{0,1\}$, where $I$ is the vertex set of $\Gamma$. Two vertices $(i,1)$ and $(j,1)$ in $I \times 1$ are connected by an edge in $\Gamma''$ if and only if the corresponding vertices $i$ and $j$ span an edge in $\Gamma$. Any two distinct vertices in $I \times 0$ are connected by an edge. Finally, vertices $(i,0)$ and $(j,1)$ are connected by an edge if and only if $i \neq j$. The vertex set of $\Gamma '$ is $I \times \{-1, 1\}$. The subsets $I\times (-1)$ and $I \times 1$ both span copies of $\Gamma$. A vertex $(i,-1)$ in $I \times (-1)$ is connected to $(j,1)$ in $I \times 1$ if and only if $i \neq j$ and the vertices $i$ and $j$ span an edge of $\Gamma$. (See above for an example.) Then

\begin{theorem}{\cite{DavJan}}\label{davjan}
$A_{\Gamma}$ and $W_{\Gamma'}$ are subgroups of $W_{\Gamma''}$ of index $2^{I}$.
\end{theorem}

\begin{lemma}\label{finiteindex}
$\Gamma$ is a join if and only if $\Gamma '$ is a join.
\end{lemma}
\begin{proof}
For any subset $I' \subset I$, let $\Gamma_{I'}$ be a full subgraph of $\Gamma$ whose vertex set is $I'$.
It is obvious that if $\Gamma$ is a join, then $\Gamma '$ is a join. Namely, if $\Gamma$ is a join of $\Gamma_{I_{1}}$ and $\Gamma_{I_{2}}$, then $\Gamma ' $ is a join of $\Gamma^{'}_{I_{1}\times\{-1,1\}}$ and $\Gamma^{'}_{I_{2}\times\{-1,1\}}$. Conversely, suppose $\Gamma'$ is a join of $\Gamma^{'}_{I'_{1}}$ and $\Gamma^{'}_{I'_{2}}$. Note that a vertex $(i,1) \in \Gamma^{'}_{I'_{1}}$ if and only if $(i,-1) \in \Gamma^{'}_{I'_{1}}$. Therefore, $\Gamma'$ is a join of $\Gamma^{'}_{I_{1}\times\{-1,1\}}$ and $\Gamma^{'}_{I_{2}\times\{-1,1\}}$ for some $I_{1}, I_{2} \subset I$ and $\Gamma$ is a join of $\Gamma_{I_1}$ and $\Gamma_{I_2}$.
\end{proof}

\begin{proof}[\,\,of Propositon \ref{art}]
Suppose that $\Gamma$ is not a join. By Lemma \ref{finiteindex}, the corresponding graph $\Gamma'$ is not a join. It follows that $W_{\Gamma '}$ is irreducible. If $W_{\Gamma'}$ has a finite index free abelian subgroup $K$, then $K \cap A_{\Gamma}$ is also a finite index free abelian subgroup of $A_{\Gamma}$. But this is impossible: Since $\Gamma$ is assumed to be not a join, there are two vertices in $\Gamma$ which are not joined by an edge and they generate a non-abelian free subgroup of $A_{\Gamma}$. Call $a$ and $b$ for the generators. On the other hand, since $K \cap A_{\Gamma}$ is of finite index in $A_{\Gamma}$, there exist $N_1$ and $N_2$ such that $a^{N_1}, b^{N_2} \in K \cap A_{\Gamma}$ and $a^{N_1} b^{N_2} =b^{N_2}a^{N_1}$. This contradicts that $a$ and $b$ generate a free group. By Proposition \ref{cox}, $\mathrm{rank}(W_{\Gamma'})=1$, and hence, $\mathrm{rank}(A_{\Gamma})=1$.
\end{proof}

\begin{corollary}
If $\Gamma$ is not a join, then $A_{\Gamma}$ is not commensurable (or quasi-isometric) to any uniform lattice in a non-compact connected semi-simple Lie group of higher rank.
\end{corollary}

%%%%%%%%%%%%%%%%%%%%%%%%%%%%%%%%%%%%%%%%%%%%%%%%%%%%%%%%%%%%%%%%%%%%%%%%%%%%%%%%%%%%%%%%%%%%%%%%%%%%%%%%%%%%%%%%%%%%%%%%%%%%%%%%%%%%%%%%%%%%%%%%%%%%%%%%%%%%%%%%%%

\section{Relatively Hyperbolic Groups}\label{relhypgppre}
Relatively hyperbolic groups are a generalization of hyperbolic groups. They were introduced by Gromov (\cite{Gro}) and many equivalent definitions have been developed by different authors in different contexts. See, for example, Farb(\cite{Far}), Bowditch(\cite{Bow}), Yaman(\cite{Yam}), Dru\c{t}u-Osin-Sapir(\cite{DruSap}), Osin(\cite{Osi}), Dru\c{t}u(\cite{Dru}), and Mineyev-Yaman(\cite{MinYam}). In this paper, we shall use the Bowditch's definition via geometrically finite convergence groups. Following \cite[Section 2]{Xie}, we recall the notion of relatively hyperbolic groups and the existence of an invariant collection of disjoint horoballs. See \cite{Bow} for details.

Suppose that $M$ is a compact metrizable topological space. Suppose that a group $G$ acts by homeomorphisms on $M$. By definition, $G$ is $\emph{a convergence group}$ if the induced action on the space of distinct triples is properly discontinuous. In such a case, we call an element $g \in G$ a \emph{hyperbolic element} if it has infinite order and fixes exactly two points in $M$. A subgroup $H$ of $G$ is \emph{parabolic} if $H$ is infinite, fixes some point in $M$, and contains no hyperbolic elements. In this case, the fixed point of $H$ is unique. We call the point a \emph{parabolic point} and the nontrivial element in a parabolic subgroup a \emph{parabolic element}. It is necessary that the stabilizer of a parabolic point $\zeta$, $\mathrm{Stab}(\zeta)$, is a parabolic subgroup. A parabolic point $\zeta$ is \emph{a bounded parabolic point} if $\mathrm{Stab}(\zeta)$ acts properly and cocompactly on $M\setminus\{\zeta\}$. A point $\xi \in M$ is \emph{a conical limit point} if there exists a sequence $\{g_{n}\}$ in $G$ and two distinct points $\zeta, \eta \in M$, such that $g_{n}(\xi) \to \zeta$ and $g_{n}(\xi') \to \eta$ for all $\xi' \neq \xi$. Finally, a convergence group $G$ on $M$ is $\emph{a geometrically finite group}$ if each point of $M$ is either a conical limit point or a bounded parabolic point.

\begin{definition}\label{relhypgp}
A group $G$ is \emph{hyperbolic relative to a family of infinite finitely generated subgroups $\mathcal{G}$} if it acts properly discontinuously by isometries on a proper geodesic hyperbolic space $X$ such that the induced action on $\partial X$ is of convergence, geometrically finite, and such that the maximal parabolic subgroups are exactly the elements of $\mathcal{G}$. Elements of $\mathcal{G}$ are called \emph{peripheral subgroups}.
\end{definition}

It is known that all the definitions mentioned above are equivalent, provided that the group $G$ and all peripheral subgroups are infinite and finitely generated. But some authors do not assume that peripheral subgroups are infinite and finitely generated. In fact, it has been shown in \cite{Yam} that the finite generation of peripheral subgroups can be dispensed with. Also some definitions allow the elements of $\mathcal{G}$ to be finite. But, in \cite{Osi}, Osin proved that one can make $\mathcal{G}$ smaller so that all peripheral subgroups are infinite (or possibly empty). The followings are well-known examples of relatively hyperbolic groups.

\begin{example}
\begin{itemize}
\item Hyperbolic groups: These are hyperbolic relative to $\mathcal{G} = \emptyset$.
\item Geometrically finite isometry groups of Hadamard manifolds of negatively pinched sectional curvature: These are hyperbolic relative to the maximal parabolic subgroups.
\item Free products of finitely many finitely generated groups: These are hyperbolic relative to the factors, since the action on the Bass-Serre tree satisfies the second definition of Bowditch. See \cite[Definition 2]{Bow}.
\item Groups $G$ acting geometrically on a CAT(0) space $X$ which has the isolated flats property: In this case, $X$ is an asymptotically tree-graded space and $G$ is hyperbolic relative to the collection of virtually abelian subgroups of rank at least two. (See \cite{HruKle})
\end{itemize}
\end{example}

In the next subsection, we will prove that if $G$ is a relatively hyperbolic group with $|\mathcal{G}| \geq 2$, and at least one peripheral subgroup contains an element of infinite order, then $\mathrm{rank}(G) \leq 1$. The following theorem of Bowditch on the existence of an invariant collection of disjoint horoballs provides the crucial tool in our proof.

Let $X$ be a $\delta$-hyperbolic geodesic metric space for some $\delta > 0$ and $\xi \in \partial X$. A function $h : X \to \mathbb{R}$ is \emph{a horofunction about $\xi$} if there exist constants $c_{1} = c_{1}(\delta), c_{2}=c_{2}(\delta)$ such that if $x,a \in X$ and $d(a,x\xi) \leq c_{1}$, for some geodesic ray $x\xi$ from $x$ to $\xi$, then $|h(a)-h(x)-d(x,a)| \leq c_{2}$. A closed set $B \subset X$ is \emph{a horoball about $\xi$} if there is a horofunction $h$ about $\xi$ and a constant $c=c(\delta)$ such that $h(x) \geq -c$ for all $x\in B$, and $h(x) \leq c$ for all $x \in X\setminus B$. In this case $\xi$ is called the center of the horoball and is uniquely determined by $B$.

\begin{proposition}{\cite[Proposition 6.13]{Bow}}\label{horo}
Let $G$ be a relatively hyperbolic group and $X$ a space on which $G$ acts as in Definition \ref{relhypgp}. Let $\Pi$ be the set of all bounded parabolic points in $\partial X$. Then $\Pi/G$ is finite. Moreover, for any $r>0$, there is a collection of horoballs $\mathcal{B}=\{B_{\xi} | \xi \in \Pi\}$ indexed by $\Pi$ with the following properties
\begin{enumerate}
\item $\mathcal{B}$ is $r$-separated, that is, $d(B_{\xi},B_{\eta}) \geq r$ for all $\xi \neq \eta \in \Pi$.
\item $\mathcal{B}$ is $G$-invariant, that is, $g(B_{\xi}) = B_{g(\zeta)}$ for all $g \in G$ and $\xi \in \Pi$.
\item $Y(\mathcal{B})/G$ is compact, where $Y(\mathcal{B}) = X \setminus \bigcup_{\xi \in \Pi} int(B_{\zeta})$.
\end{enumerate}
\end{proposition}

Note that the intersection of any two peripheral subgroups is finite and there are finitely many conjugacy classes of peripheral subgroups.

%%%%%%%%%%%%%%%%%%%%%%%%%%%%%%%%%%%%%%%%%%%%%%%%%%%%%%%%%%%%%%%%%%%%%%%%%%%%%%%%%%%%%%%%%%%%%%%%%%%%%%%%%%%%%%%%%%%%%%%%%%%%%%%%%%%%%%%%%%%%%%%%%%%%%%%%%%%%%%%%%%

\subsection{Algebraic Rank of Relatively Hyperbolic Groups}

In order to prove that a group has an algebraic rank $\leq 1$, we need to figure out the set $\mathcal{A}_{1}(G)$ and show that the group can be covered by finitely many translates of $\mathcal{A}_{1}(G)$. Also the procedure needs to be repeated for all finite index subgroups. We introduce two lemmas which enable us to find the elements such that the group can be covered by translates of $\mathcal{A}_{1}(G)$ by those elements.

For a finite set of isometries $F$ of a metric space $X$ and $x \in X$, let $\lambda(x,F) = \mathrm{max}\{d(f(x),x)|f\in F\}$.

\begin{lemma}{\cite{Kou}}\label{koubi}
Let $X$ be a $\delta$-hyperbolic geodesic metric space and $G$ a group of isometries of $X$ with a finite generating set $S$. If $\lambda(x,S) > 100\delta$ for all $x \in X$, then $G$ contains a hyperbolic element $g$ such that $d_{S}(id,g) =1$ or $2$.
\end{lemma}

Hereafter, suppose that $G$ is a relatively hyperbolic group with $|\mathcal{G}| \geq 2$ and $X$ a proper $\delta$-hyperbolic geodesic space on which $G$ acts as in Definition \ref{relhypgp}. Also we assume that there is a peripheral subgroup in $\mathcal{G}$ containing elements of infinite order. Note that the existence of such a peripheral subgroup implies that there are two or more such subgroups by conjugation by hyperbolic elements. Lemma \ref{horo} implies that there is a $200\delta$-separated invariant collection of horoballs $\mathcal{B}$ centered at the parabolic points such that $Y(\mathcal{B})/G$ is compact.

\begin{lemma}{\cite[Lemma 3.1]{Xie}}\label{xie}
There exists a positive integer $k_1$ with the following property : for any infinite order element $\gamma \in G$ and any $x \in Y(\mathcal{B})$, there is some $k, 1 \leq k \leq k_{1}$, such that $d(\gamma^{k}(x),x) \geq 200\delta$.
\end{lemma}

Let $\mathcal{H}$ be the set of hyperbolic elements.

\begin{proposition}\label{hypelerel}
$\mathcal{H} \subset \mathcal{A}_{1}(G)$.
\end{proposition}
\begin{proof}
Let $g \in \mathcal{H}$ be given and $A$ the two fixed points of $\langle g \rangle$ in $\partial X$. If $h \in G$ commutes with $g$, then $h$ fixes $A$ (See \cite[Corollary 2O]{Tuk}). Combining this with the fact that $\langle g \rangle$ is of finite index in the stabilizer $H=\{q \in G | qA=A\}$ (see \cite[Theorem 2I]{Tuk}), the centralizer of $g$ in $G$, $C_{G}(g)$ has a free abelian group of rank at most one as a finite index subgroup. Therefore, $\mathcal{H} \subset \mathcal{A}_{1}(G)$.
\end{proof}

Choose two elements of infinite order from two different peripheral subgroups and denote them by $h_{1}$ and $h_{2}$. We also denote the horoball stabilized by $h_{i}$ by $B_{i}, i=1,2$. Since $h_{i}$ is chosen to be of infinite order, $B_{i}$ is the only horoball stabilized by $h_{i}, i=1,2$. Then

\begin{proposition}\label{relmain}
Let $g \in G \setminus \mathcal{H}$ be an infinite order parabolic element. Then $h_{i}^{k}g$ is hyperbolic for some $i=1,2$ and for some $k, 1 \leq k \leq k_{1}$, where $k_{1}$ is the constant appearing in Lemma \ref{xie}.
\end{proposition}

\begin{proof}
Without loss of generality, $h_1$ and $g$ are contained in different peripheral subgroups. Following Lemma \ref{xie}, consider $K=\langle h_{1}^{k}, g \rangle$. We prove that $\lambda(x,\{h_{1}^{k},g\}) > 100 \delta$ for any $x \in X$. Suppose that there exists some $x \in B$ for some $B \in \mathcal{B}$ such that $\lambda(x,\{h_{1}^{k},g\}) \leq 100\delta$. Then $h_{1}^{k}(B)=B$ and $g(B)=B$ (Note that $\mathcal{B}$ is $200\delta$-separated). It follows that the center of $B$ is fixed by $K$. Since $h_{1}^{k}$ and $g$ fix two different centers, this is a contradiction. By the choice of $k$, $d(h_{1}^{k}(x),x) \geq 200\delta$. Therefore, $\lambda(x, \{h_{1}^{k},g\}) \geq 200\delta$. Lemma \ref{koubi} implies that $h_{1}^{k}g$ or $gh_{1}^{k}$ is hyperbolic. (Note that both $h_{1}^{k}$ and $g$ are parabolic.) Furthermore, if $gh_{1}^{k}$ is hyperbolic, so is its conjugate $g^{-1}(gh_1^{k})g = h_{1}^{k}g$. In fact, suppose that $gh_{1}^{k}$ is hyperbolic and fixes exactly two distinct points $\alpha$ and $\beta$, then $h_{1}^{k}g$ fixes $h_{1}^{k}(\alpha)$ and $h_{1}^{k}(\beta)$. Conversely, suppose $h_{1}^{k}g$ fixes a point $\gamma \neq h_{1}^{k}(\alpha), h_{1}^{k}(\beta)$. Choose $\gamma'$ such that $h_{1}^{k}(\gamma')=\gamma$. Note that $\gamma' \neq \alpha, \beta$. But $h_{1}^{k}gh_{1}^{k}(\gamma') = h_{1}^{k}(\gamma') \Rightarrow gh_{1}^{k}(\gamma')=\gamma'$. Therefore, $\gamma'=\alpha$ or $\beta$, which is a contradiction.

\end{proof}

\begin{theorem}\label{relmain2}
Suppose that $G$ is hyperbolic relative to a family $\mathcal{G}$ of infinite finitely generated subgroups. If $|\mathcal{G}| \geq 2$ and at least one subgroup in $\mathcal{G}$ contains an element of infinite order, then $\mathrm{rank}(G) \leq 1$
\end{theorem}

\begin{proof}
We decompose the set of torsion elements into $\mathcal{E} \coprod \mathcal{E'}$ as follows. A torsion element $g \in \mathcal{E}$ if and only if $g$ stabilizes both $B_{1}$ and $B_{2}$. Otherwise $g \in \mathcal{E}'$. Suppose that $g \in \mathcal{E}'$. Without loss of generality, assume that $g$ does not stabilize $B_{1}$. Then we have $\lambda(x,\{h_{1}^{k},g\}) \geq 100\delta$. In particular, $d(g(x),x) \geq 100\delta$ for $x \in B_{1}$. The same argument as in Proposition \ref{relmain} implies that $h_{1}^{k}g$ is hyperbolic. Since the intersection of two distinct peripheral subgroups is at most finite, $|\mathcal{E}| < \infty$, say $\mathcal{E} =\{l_{1}, \cdots, l_{n}\}$. Choose any hyperbolic element $h \in G$ and let $t_{i}=l_{i}h^{-1}$ for $i=1,\cdots, n$.

By combining with Proposition \ref{relmain}, 
$$G = \mathcal{H} \bigcup (\bigcup_{i=1}^{k_1} h_{1}^{-i}\mathcal{H}) \bigcup (\bigcup_{i=1}^{k_1} h_{2}^{-i}\mathcal{H}) \bigcup(\bigcup_{i=1}^{n} t_{i}\mathcal{H}).$$ Therefore, $\mathrm{r}(G) \leq 1$. 

Next we need to prove that $r(T) \leq 1$ for any finite index subgroup $T$ in $G$. By taking the normal core of $T$, it suffices to show that $r(T) \leq 1$ for any finite index normal subgroup $T$ in $G$. Recall $r(G') \geq r(G)$ if $G'$ is a finite index subgroup of $G$.
Let $T$ be a finite index normal subgroup in $G$ and $m=[G:T]$. Also let $\mathcal{H}_{T} = \mathcal{H} \bigcap T$. Recall that $\mathcal{H}$ is the set of hyperbolic elements in $G$. Then $\mathcal{H}_{T} \subset \mathcal{A}_{1}(T)$. It can be easily verified that all arguments in proving $r(G) \leq 1$ apply without any change to prove $r(T) \leq 1$, namely,

\begin{itemize}
\item $h_{i}^{m} \in T$ is an infinite order parabolic element and stabilizes a horoball $B_{i}$ for $i=1,2$.
\item For any $g \in T\setminus \mathcal{H}_{T}$ of infinite order, $(h_{i}^{m})^{k} g \in \mathcal{H}_{T}$ for some $i=1,2$.
\item One can decompose the set of torsion elements in $T$ as follows : $g \in \mathcal{E}_{T}$ if and only if $g$ stabilizes both $B_{1}$ and $B_{2}$. Otherwise $g \in \mathcal{E'}_{T}$. For any element $g$ in $\mathcal{E'}_{T}$, $(h_{i}^{m})^{k} g \in \mathcal{H}_{T}$ for some $i=1,2$. Since $\mathcal{E}_{T}$ is finite, one can choose any hyperbolic element in $T$ such that $\mathcal{E'}_{T}$ can be covered by finitely many translates of $\mathcal{H}_{T}$.
\end{itemize}
\end{proof}

%%%%%%%%%%%%%%%%%%%%%%%%%%%%%%%%%%%%%%%%%%%%%%%%%%%%%%
%%%%%%%%%%%%%%%%%%%%%%%%%%%%%%%%%%%%%%%%%%%%%%%%%%%%%%
%%%%%%%%%%%%%%%%%%%%%%%%%%%%%%%%%%%%%%%%%%%%%%%%%%%%%%

\subsection{$\mathrm{CAT}(0)$ Spaces with Isolated Flats}

$\mathrm{CAT}(0)$ spaces with isolated flats were first introduced by Kapovich-Leeb and Wise, independently. In \cite{KapLee}, Kapovich and Leeb study a class of $\mathrm{CAT}(0)$ spaces in which the maximal flats are disjoint and separated by regions of strictly negative curvature. Since then, they have been studied by a number of authors, in particular, because of their strong connections to relatively hyperbolic groups.

Throughout this subsection, a \emph{$k$-flat} is an isometrically embedded copy of Euclidean space $\mathbb{E}^{k}$ for $k \geq 2$. In particular, we don't consider a geodesic line as a flat. Let $\mathrm{Flat}(X)$ be the space of all flats in $X$ with the topology of uniform convergences on bounded sets. A $\mathrm{CAT}(0)$ space $X$ with a geometric group action has \emph{isolated flats} if it contains an equivariant collection $\mathcal{F}$ of flats such that $\mathcal{F}$ is closed and isolated in $\mathrm{Flat}(X)$, and each flat $F \subset X$ in the space is contained in a uniformly bounded tubular neighborhood of some $F' \in \mathcal{F}$. See \cite[Theorem 1.2.3]{HruKle} for equivalent formulations of $\mathrm{CAT}(0)$ spaces with isolated flats.

Let $X$ be a CAT(0) space with isolated flats and $G$ be a group acting geometrically on $X$. One of main results in \cite{HruKle} is

\begin{theorem}\cite[Theorem 1.2.1]{HruKle}\label{hrukle}
The following are equivalent.
\begin{enumerate}
\item $X$ has isolated flats.
%\item Each component of the Tits boundary $\partial_{T}X$ is either an isolated point or a standard Euclidean sphere.
\item $X$ is a relatively hyperbolic space with respect to a family of flats in $\mathcal{F}$.
\item $G$ is a relatively hyperbolic group with respect to a collection of virtually abelian subgroups of rank at least two.
\end{enumerate}
\end{theorem}

\begin{remark}
\begin{itemize}
\item In the second statement above, the term ``relatively hyperbolic" for metric spaces was introduced by Dru\c{t}u and Sapir. In \cite{DruSap}, they used the term ``asymptotically tree graded" for such spaces and proved that the metric and group theoretic notions of being relatively hyperbolic are equivalent for a finitely generated group with a word metric.
\item If $X$ has isolated flats with respect to $\mathcal{F}$, then $\mathcal{F}$ is locally finite. Combining this with the Bieberbach Theorem shows that each flat in $\mathcal{F}$ is $G$-periodic with virtually abelian stabilizer. Note that the geometric action of $G$ on $X$ induces a quasi-isometry and being relatively hyperbolic with respect to quasiflats is a geometric property. (See \cite[Theorem 5.1]{DruSap}.) This quasi-isometry takes $\mathcal{F}$ to the left cosets of a collection of virtually abelian subgroups of rank at least two. See \cite[Section 3, 4]{HruKle} for details.
\end{itemize}
\end{remark}

\begin{proposition}\label{catmain}
Let $G$ be a group acting geometrically on a $\mathrm{CAT}(0)$ space with isolated flats and $|\mathcal{F}| \geq 2$. Then
$\algrk(G) \leq 1$.
\end{proposition}

\begin{proof}
By Theorem \ref{hrukle}, $G$ is hyperbolic relative to a collection of virtually abelian subgroups of rank at least two. Since we assume that $|\mathcal{F}| \geq 2$, there are at least two peripheral subgroups in $G$. Proposition \ref{relmain2} applies that $\mathrm{rank}(G) \leq 1$.
\end{proof}

\begin{remark}
In the case that $\mathcal{F}$ consists of a single flat $F$, one can conclude that $G$ acts geometrically on $F$. Therefore, $\mathrm{rank}(G)$ is equal to $dim(F) \geq 2$ by \cite{BalEbe}.
\end{remark} 

\bibliography{referencesforproject3}{}
\bibliographystyle{plain}

\leftline{\bf Author's addresses:}

\noindent 
Raeyong Kim: 

Department of Mathematics, The Ohio State University, 
231 West 18th Avenue, Columbus, Ohio 43210-1174, United States.  

{\tt kimr@math.ohio-state.edu}

\end{document}